\documentclass[letter,10pt]{amsart}
\usepackage{amssymb,amsmath,amsfonts}
\usepackage{mathrsfs}
\usepackage[left=1 in, right=1 in,top=1 in, bottom=1 in]{geometry}
\usepackage{mathenv}

\newtheorem{theorem}{Theorem}[section]
\newtheorem{lemma}[theorem]{Lemma}

\makeatletter
\renewcommand \theequation {%
\ifnum \c@section>\z@ \@arabic\c@section.%
\fi\@arabic\c@equation} \@addtoreset{equation}{section}
\makeatother

\providecommand{\ud}[1]{\mathrm{d}{#1}}
\providecommand{\abs}[1]{\left\vert#1\right\vert}
\providecommand{\nm}[1]{\left\Vert#1\right\Vert}
\providecommand{\br}[1]{\langle #1 \rangle}

\providecommand{\id}[1]{{\bf{1}}_{\{#1\}}}
\providecommand{\lnm}[1]{\left\Vert#1\right\Vert_{L^{\infty}}}

\def\dt{\partial_t}
\def\p{\partial}

\def\half{\frac{1}{2}}
\def\rt{\rightarrow}
\def\r{\mathbb{R}}
\def\no{\nonumber}
\def\ue{\mathrm{e}}

\def\ds{\displaystyle}

\def\c{c}

\def\a{\alpha}

\def\e{\mathcal{E}}
\def\m{\mathcal{M}}
\def\ee{\mathfrak{E}}
\def\mm{\mathfrak{M}}

\def\g{\Gamma}
\def\l{\Lambda}

\begin{document}
\title{The Nonrelativistic Limit of Relativistic Vlasov-Maxwell System}

\author[J. Schaeffer]{Jack Schaeffer}
\address[J. Schaeffer]{
   \newline\indent Department of Mathematical Sciences, Carnegie Mellon University
\newline\indent Pittsburgh, PA 15213, USA}
\email{js5m@andrew.cmu.edu}

\author[L. Wu]{Lei Wu}
\address[L. Wu]{
   \newline\indent Department of Mathematical Sciences, Carnegie Mellon University
\newline\indent Pittsburgh, PA 15213, USA}
\email{lwu2@andrew.cmu.edu}

\begin{abstract}
We consider the one and one-half dimensional multi-species relativistic Vlasov-Maxwell system with non-decaying(in space) initial data. We prove its well-posedness and nonrelativistic limit as the speed of light $\c\rt\infty$. These results mainly rely on a delicate analysis of energy structure and application of estimates along the characteristic lines.\\
\textbf{Keywords:} non-decaying data; well-posedness; nonrelativistic limit.
\end{abstract}

\maketitle

\pagestyle{myheadings} \thispagestyle{plain} \markboth{JACK SCHAEFFER AND LEI WU}{NONRELATIVISTIC LIMIT OF RELATIVISTIC VLASOV-MAXWELL SYSTEM}

\section{Introduction}

Consider the one and one-half dimensional relativistic Vlasov-Maxwell system:
\begin{eqnarray}\label{simple system}
\left\{
\begin{array}{l}
\dt f^{\a}+V^{\a}_1(p)\p_{x}f^{\a}
+e^{\a}\bigg(E_1+\c^{-1}V^{\a}_2(p)B\bigg)\p_{p_1}f^{\a}+e^{\a}\bigg(E_2-\c^{-1}V^{\a}_1(p)B\bigg)\p_{p_2}f^{\a}=0,\\\rule{0ex}{1.0em}
\p_{x}E_1=4\pi\rho,\quad \dt E_1=-4\pi j_1,\\\rule{0ex}{1.0em}
\dt E_2+\c\p_{x}B=-4\pi j_2,\quad \dt B+\c\p_{x}E_2=0,\\\rule{0ex}{1.0em}
\rho(t,x)=\displaystyle\sum_{\a}\bigg(e^{\a}\int_{\r^2}f^{\a}(t,x,p)\ud{p}\bigg),\quad
j(t,x)=\displaystyle\sum_{\a}\bigg(e^{\a}\int_{\r^2}V^{\a}(p)f^{\a}(t,x,p)\ud{p}\bigg).
\end{array}
\right.
\end{eqnarray}
Here $t$ is time, $x\in\mathbb{R}$ is position, $p\in\mathbb{R}^2$ is momentum, and $f^{\alpha}(t,x,p)$ is the number density in phase space of particles of charge $e^{\alpha}$ and mass $m^{\alpha}$.  The velocity of a particle is
\begin{eqnarray}
V^{\a}(p)=\frac{p}{\sqrt{(m^{\a})^2+\c^{-2}p^2}},
\end{eqnarray}
where $c$ is the speed of light.  As defined in (\ref{simple system}) $\rho$ and $j$ are respectively the charge and current densities.  The induced electromagnetic field is given by
\begin{eqnarray}
\begin{array}{l}
\vec{E}=(E_1,E_2,0),\quad \vec{B}=(0,0,B).
\end{array}
\end{eqnarray}
As initial data for (\ref{simple system}) we take
\begin{eqnarray}\label{initial data}
\left\{
\begin{array}{l}
f^{\a}(t,x,p)=f_0^{\a}(x,p),\\
E_1(0,x)=E_{1,0}(x),\\
E_2(0,x)=E_{2,0}(x),\\
B(0,x)=B_{0}(x)
\end{array}
\right.
\end{eqnarray}
to be given, where it is assumed that
\begin{eqnarray}\label{wt j4}
\p_{x}E_{1,0}=4\pi\displaystyle\sum_{\a}\bigg(e^{\a}\int_{\r^2}f^{\a}_0(x,p)\ud{p}\bigg).
\end{eqnarray}

The goal of this work is to study the behavior of $f^{\alpha}, E_1, E_2, B$ as $c\rightarrow\infty$.  There are several papers in the literature ~\cite{1,2,8,17,28,29,27} that study this limit for solutions of the two and three dimensional versions of (\ref{simple system}) where $f^{\alpha}\rightarrow 0$ as $|x|\rightarrow\infty$.  The goal here is to consider solutions that do not decay as $|x|\rightarrow\infty$ and hence have infinite charge and energy.  As $c\rightarrow\infty$ the limiting problem is the Vlasov-Poisson system where the lack of spatial decay is a serious issue ~\cite{3,4,5,6,7,16,19,20,21,22,25,26}.  Thus assumptions must be made on the large $|x|$ behavior of the initial data.  As in ~\cite{25} assume that for each $\alpha$ there are $F^{\alpha}:\mathbb{R}^2\rightarrow[0,\infty)$ which is $C^1$ and positive constants $R_0$ and $Q_0$ such that $|x|\geq R_0$ implies
\begin{eqnarray}\label{wt j2}
f^{\alpha}_0(x,p)-F^{\alpha}(p)=E_{2,0}(x)=B_0(x)=0
\end{eqnarray}
and $|p|\geq Q_0$ implies
\begin{eqnarray}\label{wt j3}
f^{\alpha}_0(x,p)=F^{\alpha}(p)=0.
\end{eqnarray}
Further assume that
\begin{eqnarray}\label{wt j6}
\sum_{\a}\bigg(e^{\a}\int_{\r^2}F^{\a}(p)\ud{p}\bigg)&=&0,\\\label{wt j7}
\sum_{\a}\bigg(e^{\a}\int_{\r^2}F^{\a}(p)\frac{p}{m^{\a}}\ud{p}\bigg)&=&0,
\end{eqnarray}
and
\begin{eqnarray}\label{wt j8}
\sum_{\a}\Bigg(e^{\a}\int_{\r}\int_{\r^2}\bigg(f^{\a}_0(x,p)-F^{\a}(p)\bigg)\ud{p}\ud{x}\Bigg)&=&0.
\end{eqnarray}
Then define
\begin{eqnarray}\label{wt j9}
\rho_0(x)&=&\sum_{\a}\bigg(e^{\a}\int_{\r^2}f^{\a}_0(x,p)\ud{p}\bigg),\\\label{wt j10}
j_0(x)&=&\sum_{\a}\bigg(e^{\a}\int_{\r^2}f^{\a}_0(x,p)V^{\a}(p)\ud{p}\bigg),
\end{eqnarray}
and
\begin{eqnarray}\label{wt j5}
E_{1,0}(x)=2\pi\int_{-\infty}^x\rho_0(y)\ud{y}-2\pi\int^{\infty}_x\rho_0(y)\ud{y}.
\end{eqnarray}
Note that (\ref{wt j4}) follows from (\ref{wt j5})
and for any $\abs{x}\geq R_0$,
\begin{eqnarray}
E_{1,0}(x)=0.
\end{eqnarray}
Then the main results of this paper are the following two theorems:

\begin{theorem}\label{main 1}
Let $f^{\a}_0\geq 0$, $E_{2,0}$, and $B_0$ be $C^2$ and assume that (\ref{wt j6}) through (\ref{wt j5}) hold.  Then there is a global $C^1$ solution $(f^{\a},E_1,E_2,B)$ of (\ref{simple system}) and (\ref{initial data}).  Moreover, for every $T>0$ and $c\geq 1$ there exists $C_0>0$ (depending on $T$ and initial data, but not on $c$) such that
\begin{eqnarray}
|f^{\a}(t,x,p)|+|E_1(t,x)|+|E_2(t,x)|+|B(t,x)|&\leq&C_0
\end{eqnarray}
for every $\a$ and every $(t,x,p)\in [0,T]\times\mathbb{R}\times\mathbb{R}^2$.
\end{theorem}

To state the second theorem we must define $f^{\a,\infty}$ and $E^{\infty}_1$ by
\begin{eqnarray}\label{limit system}
\left\{
\begin{array}{l}\label{wt j11}
\dt f^{\a,\infty}+V^{\a,\infty}_1(p)\p_{x}f^{\a,\infty}
+e^{\a}E_1^{\infty}\p_{p_1}f^{\a,\infty}=0,\\\rule{0ex}{1.0em}
\ds\rho^{\infty}=\sum_{\a}e^{\a}\int_{\mathbb{R}^2}f^{\a,\infty}dp,\\\rule{0ex}{1.0em}
\ds j^{\infty}=\sum_{\a}e^{\a}\int_{\mathbb{R}^2}f^{\a,\infty}\frac{p}{m^{\a}}dp,\\\rule{0ex}{1.0em}
\ds E^{\infty}_1=2\pi\int_{\infty}^x\rho^{\infty}(y)dy-2\pi\int_x^{\infty}\rho^{\infty}(y)dy,
\end{array}
\right.
\end{eqnarray}
with
\begin{eqnarray}\label{wt j12}
f^{\a,\infty}(0,x,p)=f^{\a}_0(x,p)
\end{eqnarray}
and
\begin{eqnarray}
V^{\a,\infty}(p)=\frac{p}{m^{\a}}.
\end{eqnarray}
From ~\cite{19} it is known that (\ref{wt j11}) and (\ref{wt j12}) possesses a global $C^1$ solution.

\begin{theorem}\label{main 2}
Assume that
\begin{eqnarray}
E_{2,0}=B_0=0.
\end{eqnarray}
Then, with the same assumptions as in Theorem \ref{main 1}, for every $T>0$ and $c\geq 1$ there exists $C_0>0$ (depending on $T$ and initial data, but not on $c$) such that
\begin{eqnarray}
|f^{\a}(t,x,p)-f^{\a,\infty}(t,x,p)|+|E_1(t,x)-E_1^{\infty}(t,x)|+|E_2(t,x)|+|B(t,x)|&\leq&C_0\c^{-1}
\end{eqnarray}
for every $\a$ and every $(t,x,p)\in [0,T]\times\mathbb{R}\times\mathbb{R}^2$.
\end{theorem}

In the following, we will use $\lnm{\cdot}$ to represent $L^{\infty}$ norm either in $(t,x,p)\in[0,T]\times\r\times\r^2$, $(t,x)\in[0,T]\times\r$, or $(x,p)\in\r\times\r^2$. We will use $C_0$ to indicate a universal constant, which may change from line to line and may depend on $T$ and initial data, but not on $c$.

The global existence of smooth solutions of (\ref{simple system}) was established in ~\cite{10}.  This was extended to two dimensions in ~\cite{11} and ~\cite{12}, but remains open in three dimensions.  However, it was shown in ~\cite{13} that solutions of the three dimensional problem can break down only if particle speeds approach the speed of light.  The global existence of smooth solutions for the Vlasov-Poisson system is better understood, see ~\cite{15,18,23,24}.  Also see ~\cite{9} for a general reference on mathematical kinetic theory.

It has been suggested ~\cite{14} that when studying a nonrelativistic limit it is desirable to keep the speed of light constant and analyze the limiting behavior in some other parameter.  While this framework is appealing, it was not clear what other parameter to use that would not complicate both the analysis and comparison with papers such as ~\cite{1,2,8,17,28,29,27}.

This paper is organized as follows:  The proof of Theorem \ref{main 1} is in section 2.  Section 3 contains the proof of Theorem \ref{main 2}. The assumptions of Theorem \ref{main 1} are in force in section 2 and the assumptions of Theorem \ref{main 2} are in force in section 3.

Let us also define the characteristics of the Vlasov equation $(X^{\a}(s;t,x,p),P^{\a}(s;t,x,p))$ by

\begin{eqnarray}
\frac{dX^{\a}}{ds}&=&V^{\a}_1(P^{\a}),\\\
\frac{dP^{\a}_1}{ds}&=&E_1(s,X^{\a})+c^{-1}V^{\a}_2(P^{\a})B(s,X^{\a}),\\
\frac{dP^{\a}_2}{ds}&=&E_2(s,X^{\a})-c^{-1}V^{\a}_1(P^{\a})B(s,X^{\a}),
\end{eqnarray}
with $X^{\a}(t,t,x,p)=x,$ and $P^{\a}(t,t,x,p)=p$.

\section{Well-Posedness of the Relativistic Vlasov-Maxwell System}

The global existence stated in Theorem \ref{main 1} follows from the global existence result of ~\cite{10} by a finite speed of propagation argument.  Note that it was assumed in ~\cite{10} that the initial data had compact support.  So to construct the solution on $(t,x)\in [0,T]\times[-L,L]$ (with $L>R_0$) with initial data $f^{\a}_0$, $E_{2,0}$, $B_0$ as in Theorem \ref{main 1}, let $\overline{f}^{\a}_0,\overline{E}_{2,0},\overline{B}_0$ be smooth and satisfy
\begin{eqnarray}
\overline{f}^{\a}_0=f^{\a}_0,\quad \overline{E}_{2,0}=E_{2,0}\quad\overline{B}_0=B_0
\end{eqnarray}
if $|x|\leq L+cT$,
\begin{eqnarray}
\overline{f}^{\a}_0=\overline{E}_{2,0}=\overline{B}_0=0
\end{eqnarray}
if $|x|\geq L+cT+1$ and
\begin{eqnarray}
\sum_{\a}e^{\a}\int_{\mathbb{R}^2}\overline{f}^{\a}_0dp=0
\end{eqnarray}
if $|x|\geq L+cT$.  By ~\cite{10} (\ref{simple system}) possesses a global $C^1$ solution $\overline{f}^{\a},\overline{E}_1,\overline{E}_2,\overline{B}$ with initial data $\overline{f}^{\a}_0,\overline{E}_{2,0},\overline{B}_0$.  Since increasing $L$ and $T$ will not change $\overline{f}^{\a},\overline{E}_1,\overline{E}_2,\overline{B}$ on the set $\{(t,x):0\leq t$ and $|x|\leq L+c(T-t)\}$, it follows that
\begin{eqnarray}
(f^{\a},E_1,E_2,B)=\lim_{L,T\rightarrow\infty}(\overline{f}^{\a},\overline{E}_1,\overline{E}_2,\overline{B})
\end{eqnarray}
is a global smooth solution of (\ref{simple system}).

Note that by use of the characteristics of the Vlasov equation it follows that
\begin{eqnarray}
0\leq f^{\a}\leq \max{f^{\a}_0}\leq C_0.
\end{eqnarray}

\subsection{Estimate of $E_2$ and $B$}

\begin{lemma}\label{lemma 1}
We have
\begin{eqnarray}
\lnm{E_2}+\lnm{B}&\leq&C_0
\end{eqnarray}
\end{lemma}
\begin{proof}
We divide the proof into several steps:\\
\ \\
Step 1: Relativistic energy estimate.\\
We multiply $\sqrt{(m^{\a})^2+\c^{-2}p^2}$ on both sides of the Vlasov equation, sum up over $\a$, and integrate over $p\in\r^2$ to obtain
\begin{eqnarray}\label{wt 01}
\sum_{\a}\int_{\r^2}\dt f^{\a}\sqrt{(m^{\a})^2+\c^{-2}p^2}\ud{p}+\sum_{\a}\int_{\r^2}p_1\p_xf^{\a}\ud{p}\\
+\sum_{\a}\int_{\r^2}e^{\a}\left(E_1\sqrt{(m^{\a})^2+\c^{-2}p^2}+\c^{-1}p_2B\right)\p_{p_1}f^{\a}\ud{p}\no\\
+\sum_{\a}\int_{\r^2}e^{\a}\left(E_2\sqrt{(m^{\a})^2+\c^{-2}p^2}-\c^{-1}p_1B\right)\p_{p_2}f^{\a}\ud{p}&=&0.\no
\end{eqnarray}
Integrating by parts in (\ref{wt 01}), we get
\begin{eqnarray}
\dt\bigg(\sum_{\a}\int_{\r^2}f^{\a}\sqrt{(m^{\a})^2+\c^{-2}p^2}\ud{p}\bigg)+\p_x\bigg(\sum_{\a}\int_{\r^2}p_1f^{\a}\ud{p}\bigg)\\
-\sum_{\a}\int_{\r^2}e^{\a}E_1\frac{\c^{-2}p_1}{\sqrt{(m^{\a})^2+\c^{-2}p^2}}f^{\a}\ud{p}
-\sum_{\a}\int_{\r^2}e^{\a}E_2\frac{\c^{-2}p_2}{\sqrt{(m^{\a})^2+\c^{-2}p^2}}f^{\a}\ud{p}&=&0,\no
\end{eqnarray}
which further implies
\begin{eqnarray}\label{wt 02}
\dt\bigg(\sum_{\a}\int_{\r^2}f^{\a}\sqrt{(m^{\a})^2+\c^{-2}p^2}\ud{p}\bigg)+\p_x\bigg(\sum_{\a}\int_{\r^2}p_1f^{\a}\ud{p}\bigg)\\
-E_1\sum_{\a}\int_{\r^2}e^{\a}\c^{-2}f^{\a}V^{\a}_1(p)\ud{p}
-E_2\sum_{\a}\int_{\r^2}e^{\a}\c^{-2}f^{\a}V^{\a}_2(p)\ud{p}&=&0.\no
\end{eqnarray}
Based on the definition of $j(t,x)$, from (\ref{wt 02}), we deduce that
\begin{eqnarray}\label{wt 03}
\dt\bigg(\c^2\sum_{\a}\int_{\r^2}f^{\a}\sqrt{(m^{\a})^2+\c^{-2}p^2}\ud{p}\bigg)+\p_x\bigg(\c^2\sum_{\a}\int_{\r^2}f^{\a}p_1\ud{p}\bigg)
-(E_1j_1+E_2j_2)&=&0.
\end{eqnarray}
Multiplying $E_1$, $E_2$ and $B$ on the corresponding Maxwell equations, we obtain
\begin{eqnarray}
E_1\dt E_1&=&-4\pi E_1j_1,\\
E_2\dt E_2+\c E_2\p_{x}B&=&-4\pi E_2j_2,\\
B\dt B+\c B\p_{x}E_2&=&0.
\end{eqnarray}
Summing them up yields
\begin{eqnarray}\label{wt 04}
\half\dt(E_1^2+E_2^2+B^2)+\c \p_x(E_2B)=-4\pi\bigg(E_1j_1+E_2j_2\bigg).
\end{eqnarray}
Substituting (\ref{wt 04}) into (\ref{wt 03}), we have
\begin{eqnarray}\label{wt 05}
\dt\bigg(\c^2\sum_{\a}\int_{\r^2}f^{\a}\sqrt{(m^{\a})^2+\c^{-2}p^2}\ud{p}\bigg)+\p_x\bigg(\c^2\sum_{\a}\int_{\r^2}f^{\a}p_1\ud{p}\bigg)\\
+\frac{1}{8\pi}\dt\bigg(E_1^2+E_2^2+B^2\bigg)+\frac{\c}{4\pi}\p_x(E_2B)&=&0.\no
\end{eqnarray}
Based on (\ref{wt 05}), we define
\begin{eqnarray}
\e&=&\c^2\bigg(\sum_{\a}\int_{\r^2}f^{\a}\sqrt{(m^{\a})^2+\c^{-2}p^2}\ud{p}\bigg)+\frac{1}{8\pi}\bigg(E_1^2+E_2^2+B^2\bigg),\\
\m&=&\c^2\bigg(\sum_{\a}\int_{\r^2}f^{\a}p_1\ud{p}\bigg)+\frac{\c}{4\pi}(E_2B),
\end{eqnarray}
which satisfies
\begin{eqnarray}\label{wt j1}
\dt\e+\p_x\m=0.
\end{eqnarray}
In [10] (\ref{wt j1}) was the crucial ingredient.  Since we need bounds independent of $c$ here, we further define
\begin{eqnarray}
\ee&=&\e-\c^2\bigg(\sum_{\a}\int_{\r^2}m^{\a}f^{\a}\ud{p}\bigg),\\
\mm&=&\m-\c^2\bigg(\sum_{\a}\int_{\r^2}m^{\a}f^{\a}V^{\a}_1(p)\ud{p}\bigg),
\end{eqnarray}
and use the Vlasov equation and integration by parts to verify that
\begin{eqnarray}\label{wt 06}
\dt\ee+\p_x\mm&=&(\dt\e+\p_x\m)-\c^2\sum_{\a}\int_{\r^2}m^{\a}\bigg(\dt f^{\a}+V^{\a}_1(p)\p_xf^{\a}\bigg)\ud{p}\\
&=&\c^2\sum_{\a}\int_{\r^2}m^{\a}\bigg(e^{\a}\left(E_1+\c^{-1}V^{\a}_2(p)B\right)\p_{p_1}f^{\a}+e^{\a}\left(E_2-\c^{-1}V^{\a}_1(p)B\right)\p_{p_2}f^{\a}\bigg)\ud{p}\no\\
&=&-\c^2\sum_{\a}\int_{\r^2}m^{\a}\bigg(-e^{\a}\frac{\c^{-3}Bp_1p_2}{(\sqrt{(m^{\a})^2+\c^{-2}p^2})^3}f^{\a}
+e^{\a}\frac{\c^{-3}Bp_1p_2}{(\sqrt{(m^{\a})^2+\c^{-2}p^2})^3}f^{\a}\bigg)\ud{p}\no\\
&=&0.\no
\end{eqnarray}
\ \\
Step 2: Characteristic triangle.\\
We consider the point $(t,x)\in[0,T]\times\r$ in the time-space plane and the triangle bounded by $\tau=0$, $y=x-\c(t-\tau)$ and $y=x+\c(t-\tau)$ for $\tau\in[0,t]$. Integrating (\ref{wt 06}) over this triangular region and applying the divergence theorem we find that
\begin{eqnarray}\label{wt 07}
0&=&\int_0^t\int_{x-\c(t-\tau)}^{x+\c(t-\tau)}(\p_{\tau}\ee+\p_y\mm)\ud{y}\ud{\tau}=L+M+N,
\end{eqnarray}
where
\begin{eqnarray}
L&=&\int_{x-\c t}^{x+\c t}\bigg(\br{\ee,\mm}\bigg|_{(0,y)}\cdot\br{-1,0}\bigg)\ud{y},\\
M&=&\int_0^t\bigg(\br{\ee,\mm}\bigg|_{(\tau,x+\c(t-\tau))}\cdot\frac{\br{\c,1}}{\sqrt{1+\c^2}}\bigg)\sqrt{1+\c^2}\ud{\tau},\\
N&=&\int_0^t\bigg(\br{\ee,\mm}\bigg|_{(\tau,x-\c(t-\tau))}\cdot\frac{\br{\c,-1}}{\sqrt{1+\c^2}}\bigg)\sqrt{1+\c^2}\ud{\tau}.
\end{eqnarray}
Simplifying (\ref{wt 07}), we obtain
\begin{eqnarray}
\int_{x-\c t}^{x+\c t}\ee(0,y)\ud{y}&=&\int_0^t\bigg(\c\ee+\mm\bigg)(\tau,x+\c(t-\tau))\ud{\tau}+\int_0^t\bigg(\c\ee-\mm\bigg)(\tau,x-\c(t-\tau))\ud{\tau},
\end{eqnarray}
which further implies
\begin{eqnarray}\label{wt 08}
\\
\c^{-1}\int_{x-\c t}^{x+\c t}\ee(0,y)\ud{y}&=&\int_0^t\bigg(\ee+\c^{-1}\mm\bigg)(\tau,x+\c(t-\tau))\ud{\tau}+\int_0^t\bigg(\ee-\c^{-1}\mm\bigg)(\tau,x-\c(t-\tau))\ud{\tau}.\no
\end{eqnarray}
We can simply denote (\ref{wt 08}) as $I=II+III$. Then we need to estimate $I$, $II$ and $III$.\\
\ \\
Step 3: Estimate of $I$.\\
Note that
\begin{eqnarray}
\ee&=&\c^2\bigg(\sum_{\a}\int_{\r^2}f^{\a}\sqrt{(m^{\a})^2+\c^{-2}p^2}\ud{p}\bigg)+\frac{1}{8\pi}\bigg(E_1^2+E_2^2+B^2\bigg)
-\c^2\bigg(\sum_{\a}\int_{\r^2}m^{\a}f^{\a}\ud{p}\bigg)\\
&=&\c^2\bigg(\sum_{\a}\int_{\r^2}f^{\a}\left(\sqrt{(m^{\a})^2+\c^{-2}p^2}-m^{\a}\right)\ud{p}\bigg)+\frac{1}{8\pi}\bigg(E_1^2+E_2^2+B^2\bigg)\no\\
&=&\c^2\Bigg(\sum_{\a}\int_{\r^2}f^{\a}\left(\frac{\c^{-2}p^2}{\sqrt{(m^{\a})^2+\c^{-2}p^2}+m^{\a}}\right)\ud{p}\Bigg)+\frac{1}{8\pi}\bigg(E_1^2+E_2^2+B^2\bigg)\no\\
&\leq&\Bigg(\sum_{\a}\frac{1}{2m^{\a}}\int_{\r^2}\abs{f^{\a}}p^2\ud{p}\Bigg)+\frac{1}{8\pi}\bigg(E_1^2+E_2^2+B^2\bigg).\no
\end{eqnarray}
Therefore, we have
\begin{eqnarray}\label{wt 10}
I&=&\c^{-1}\int_{x-\c t}^{x+\c t}\ee(0,y)\\
&\leq&\c^{-1}\int_{x-\c t}^{x+\c t}\Bigg(\sum_{\a}\frac{1}{2m^{\a}}\int_{\r^2}\abs{f^{\a}_0}p^2\ud{p}\Bigg)\ud{y}+\c^{-1}\int_{x-\c t}^{x+\c t}\frac{1}{8\pi}\bigg(E_{1,0}^2+E_{2,0}^2+B_0^2\bigg)\ud{y}\no\\
&\leq&2t\max_{x}\Bigg\{\sum_{\a}\frac{1}{2m^{\a}}\int_{\r^2}\abs{f^{\a}_0}p^2\ud{p}\Bigg\}+\frac{t}{4\pi}\max_{x}\bigg\{E_{1,0}^2+E_{2,0}^2+B_0^2\bigg\}\no\\
&\leq&C_0,\no
\end{eqnarray}
which is uniform in $\c$.\\
\ \\
Step 4: Estimate of $II$ and $III$.\\
Note that
\begin{eqnarray}\label{wt 09}
\\
\ee\pm\c^{-1}\mm&=&\Bigg(\c^2\bigg(\sum_{\a}\int_{\r^2}f^{\a}\sqrt{(m^{\a})^2+\c^{-2}p^2}\ud{p}\bigg)+\frac{1}{8\pi}\bigg(E_1^2+E_2^2+B^2\bigg)
-\c^2\bigg(\sum_{\a}\int_{\r^2}m^{\a}f^{\a}\ud{p}\bigg)\Bigg)\no\\
&&\pm\c^{-1}\Bigg(\c^2\bigg(\sum_{\a}\int_{\r^2}f^{\a}p_1\ud{p}\bigg)+\frac{\c}{4\pi}(E_2B)-\c^2\bigg(\sum_{\a}\int_{\r^2}m^{\a}f^{\a}V^{\a}_1(p)\ud{p}\bigg)\Bigg)\no\\
&=&\Bigg(\c^2\sum_{\a}\int_{\r^2}f^{\a}\left(\sqrt{(m^{\a})^2+\c^{-2}p^2}-m^{\a}\pm\c^{-1}(p_1-m^{\a}V_1^{\a})\right)\ud{p}\Bigg)\no\\
&&+\frac{1}{8\pi}\bigg(E_1^2+E_2^2+B^2\pm 2E_2B\bigg)\no\\
&=&\Bigg(\c^2\sum_{\a}\int_{\r^2}f^{\a}\left(\sqrt{(m^{\a})^2+\c^{-2}p^2}-m^{\a}\right)\left(1\pm\c^{-1}V^{\a}_1\right)\ud{p}\Bigg)+\frac{1}{8\pi}\bigg(E_1^2+(E_2\pm B)^2\bigg).\no
\end{eqnarray}
Since
\begin{eqnarray}
E_1^2+(E_2\pm B)^2\geq0,
\end{eqnarray}
\begin{eqnarray}
\abs{\c^{-1}V^{\a}_1}=\abs{\frac{\c^{-1}p_1}{\sqrt{(m^{\a})^2+\c^{-2}p^2}}}\leq \abs{\frac{\c^{-1}p}{\sqrt{(m^{\a})^2+\c^{-2}p^2}}}<1,
\end{eqnarray}
\begin{eqnarray}
\sqrt{(m^{\a})^2+\c^{-2}p^2}-m^{\a}\geq0,
\end{eqnarray}
and $f^{\a}\geq0$, we deduce
\begin{eqnarray}
\ee\pm\c^{-1}\mm\geq\c^2\sum_{\a}\int_{\r^2}f^{\a}\left(\sqrt{(m^{\a})^2+\c^{-2}p^2}-m^{\a}\right)\left(1\pm\c^{-1}V^{\a}_1\right)\ud{p}\geq0.
\end{eqnarray}
Define
\begin{eqnarray}
\g^{\a}=\sqrt{(m^{\a})^2+\c^{-2}p^2}.
\end{eqnarray}
Then based on (\ref{wt 09}), we have
\begin{eqnarray}
II&=&\int_0^t\bigg(\ee+\c^{-1}\mm\bigg)(\tau,x+\c(t-\tau))\ud{\tau}\label{wt 11}\\
&\geq&\int_0^t\Bigg(\c^2\sum_{\a}\int_{\r^2}f^{\a}\left(\g^{\a}-m^{\a}\right)\left(1+\c^{-1}V^{\a}_1\right)\ud{p}\Bigg)(\tau,x+\c(t-\tau))\ud{\tau}\geq0,\no\\
III&=&\int_0^t\bigg(\ee-\c^{-1}\mm\bigg)(\tau,x-\c(t-\tau))\ud{\tau}\label{wt 12}\\
&\geq&\int_0^t\Bigg(\c^2\sum_{\a}\int_{\r^2}f^{\a}\left(\g^{\a}-m^{\a}\right)\left(1-\c^{-1}V^{\a}_1\right)\ud{p}\Bigg)(\tau,x-\c(t-\tau))\ud{\tau}\geq0.\no
\end{eqnarray}
\ \\
Step 5: Synthesis.\\
Collecting the results in (\ref{wt 10}), (\ref{wt 11}) and (\ref{wt 12}) in (\ref{wt 08}), we obtain
\begin{eqnarray}
0&\leq&\int_0^tk_{+}(\tau,x+\c(t-\tau))\ud{\tau}\leq C_0,\label{wt 17}\\
0&\leq&\int_0^tk_{-}(\tau,x-\c(t-\tau))\ud{\tau}\leq C_0,\label{wt 18}
\end{eqnarray}
where
\begin{eqnarray}
k_{\pm}=\c^2\sum_{\a}\int_{\r^2}f^{\a}\left(\g^{\a}-m^{\a}\right)\left(1\pm\c^{-1}V^{\a}_1\right)\ud{p}.
\end{eqnarray}
This is the starting point for the following estimates. For convenience, we denote
\begin{eqnarray}
\sigma_{\pm}^{\a}=\c^2\left(\g^{\a}-m^{\a}\right)\left(1\pm\c^{-1}V^{\a}_1\right),
\end{eqnarray}
so that
\begin{eqnarray}
k_{\pm}=\sum_{\a}\int_{\r^2}f^{\a}\sigma_{\pm}^{\a}\ud{p}.
\end{eqnarray}
\ \\
Step 6: Estimate of $j_2$.\\
Since
\begin{eqnarray}
\sigma_{\pm}^{\a}&=&\c^2\left(\frac{\c^{-2}p^2}{\g^{\a}+m^{\a}}\right)\left(1\pm\c^{-1}V^{\a}_1\right)=
\left(\frac{p^2}{\g^{\a}+m^{\a}}\right)\left(1\pm\c^{-1}V^{\a}_1\right),
\end{eqnarray}
we directly estimate
\begin{eqnarray}
\sigma_{\pm}^{\a}&=&\left(\frac{p^2}{\sqrt{(m^{\a})^2+\c^{-2}p^2}+m^{\a}}\right)\left(\frac{\sqrt{(m^{\a})^2+\c^{-2}p^2}\pm\c^{-1}p_1}{\sqrt{(m^{\a})^2+\c^{-2}p^2}}\right)\\
&\geq&\left(\frac{p^2}{\sqrt{(m^{\a})^2+\c^{-2}p^2}+m^{\a}}\right)\left(\frac{\sqrt{(m^{\a})^2+\c^{-2}p^2}-\c^{-1}\abs{p_1}}{\sqrt{(m^{\a})^2+\c^{-2}p^2}}\right)\no\\
&=&\left(\frac{p^2}{\sqrt{(m^{\a})^2+\c^{-2}p^2}+m^{\a}}\right)
\left(\frac{(m^{\a})^2+\c^{-2}p^2_2}{\left(\sqrt{(m^{\a})^2+\c^{-2}p^2}\right)\left(\sqrt{(m^{\a})^2+\c^{-2}p^2}+\c^{-1}\abs{p_1}\right)}\right)\no\\
&\geq&\frac{p^2\left((m^{\a})^2+\c^{-2}p^2_2\right)}{4\left(\sqrt{(m^{\a})^2+\c^{-2}p^2}\right)^3}\no\\
&=&\frac{p^2\left((m^{\a})^2+\c^{-2}p^2_2\right)}{4(\g^{\a})^3}.\no
\end{eqnarray}
$\sigma_{\pm}^{\a}$ builds the bridge between $j_2$ and $k_{\pm}$. In detail, we can decompose
\begin{eqnarray}\label{wt 16}
j_2&=&\sum_{\a}\bigg(e^{\a}\int_{\r^2}V^{\a}_2f^{\a}(t,x,p)\ud{p}\bigg)\\
&=&\sum_{\a}\bigg(e^{\a}\int_{\abs{p}\geq\c}V^{\a}_2f^{\a}(t,x,p)\ud{p}\bigg)+\sum_{\a}\bigg(e^{\a}\int_{1\leq\abs{p}\leq\c}V^{\a}_2f^{\a}(t,x,p)\ud{p}\bigg)\no\\
&&+\sum_{\a}\bigg(e^{\a}\int_{\abs{p}\leq1}V^{\a}_2f^{\a}(t,x,p)\ud{p}\bigg).\no
\end{eqnarray}
Here we estimate these three terms separately:
\begin{itemize}
\item
For $\abs{p}\geq\c$, we have
\begin{eqnarray}
\sigma_{\pm}^{\a}&\geq&\frac{p^2\left((m^{\a})^2+\c^{-2}p^2_2\right)}{4(\g^{\a})^3}\geq \frac{p^22m^{\a}c^{-1}|p_2|}{4(\g^{\a})^3}\\
&=&\frac{1}{2}m^{\a}c|V^{\a}_2|\frac{c^{-2}p^2}{(m^{\a})^2+c^{-2}p^2}\geq \frac{1}{2}m^{\a}c|V^{\a}_2|\frac{1}{(m^{\a})^2+1}\geq C_0|V^{\a}_2|.\no
\end{eqnarray}
Hence, we know
\begin{eqnarray}\label{wt 13}
\abs{\sum_{\a}\bigg(e^{\a}\int_{\abs{p}\geq\c}V^{\a}_2f^{\a}(t,x,p)\ud{p}\bigg)}&\leq&
{\sum_{\a}\bigg(|e^{\a}|\int_{\abs{p}\geq\c}\frac{\sigma_{\pm}^{\a}}{C_0}f^{\a}(t,x,p)\ud{p}\bigg)}
\leq C_0k_{\pm}.
\end{eqnarray}

\item
For $1\leq\abs{p}\leq\c$, we have
\begin{eqnarray}
\sigma_{\pm}^{\a}&\geq&\frac{p^2(m^{\a})^2}{4(\g^{\a})^3}\geq \frac{\abs{p_2}}{4\g^{\a}}\frac{{|p|}(m^{\a})^2}{(\g^{\a})^2}\geq \frac{1}{4}|V^{\a}_2|\frac{(m^{\a})^2}{(m^{\a})^2+1}.
\end{eqnarray}
Hence, we know
\begin{eqnarray}\label{wt 14}
\abs{\sum_{\a}\bigg(e^{\a}\int_{1\leq\abs{p}\leq\c}V^{\a}_2f^{\a}(t,x,p)\ud{p}\bigg)}&\leq&
4{\sum_{\a}\bigg(|e^{\a}|\int_{1\leq\abs{p}\leq\c}\frac{(m^{\a})^2+1}{(m^{\a})^2}\sigma_{\pm}^{\a}f^{\a}(t,x,p)\ud{p}\bigg)}\\
&\leq&4\max_{\a}|e^{\a}|{\dfrac{(m^{\a})^2+1}{(m^{\a})^2}}k_{\pm}\leq C_0k_{\pm}.\no
\end{eqnarray}

\item
For $\abs{p}\leq1$, we know
\begin{eqnarray}\label{wt 15}
\abs{\sum_{\a}\bigg(e^{\a}\int_{\abs{p}\leq1}V^{\a}_2f^{\a}(t,x,p)\ud{p}\bigg)}&\leq&
\max_{\a}{\dfrac{1}{m^{\a}}}{\sum_{\a}\bigg(|e^{\a}|\int_{\abs{p}\leq1}f^{\a}(t,x,p)\ud{p}\bigg)}\\
&\leq&C_0{\sum_{\a}\bigg(|e^{\a}|\int_{\abs{p}\leq1}\ud{p}\bigg)}\leq C_0.\no
\end{eqnarray}

\end{itemize}
Collecting the results in (\ref{wt 16}), (\ref{wt 13}), (\ref{wt 14}) and (\ref{wt 15}), we have
\begin{eqnarray}
\abs{j_2}\leq C_0(1+k_{\pm}).
\end{eqnarray}
\ \\
Step 7: Estimate of $E_2$ and $B$.\\
From Maxwell equations, we know
\begin{eqnarray}
\dt E_2+\c\p_{x}B&=&-4\pi j_2,\\
\dt B+\c\p_{x}E_2&=&0.
\end{eqnarray}
Therefore, we have
\begin{eqnarray}\label{wt j17}
\dt(E_2\pm B)\pm\c\p_x(E_2\pm B)=-4\pi j_2.
\end{eqnarray}
Hence, we have
\begin{eqnarray}
(E_2\pm B)(t,x)&=&(E_2\pm B)(0,x\mp\c t)-4\pi\int_0^tj_2(\tau,x\mp\c(t-\tau))\ud{\tau},
\end{eqnarray}
which further implies
\begin{eqnarray}
\abs{(E_2\pm B)(t,x)}&\leq&\abs{(E_2\pm B)(0,x\mp\c t)}+4\pi\int_0^t\abs{j_2}(\tau,x\mp\c(t-\tau))\ud{\tau}\\
&\leq&C_0+C_0\int_0^t(1+k_{\pm})(\tau,x\mp\c(t-\tau))\ud{\tau}.
\end{eqnarray}
Based on (\ref{wt 17}) and (\ref{wt 18}), we conclude
\begin{eqnarray}
\abs{(E_2\pm B)(t,x)}&\leq&C_0,
\end{eqnarray}
which further implies
\begin{eqnarray}
\lnm{E_2}&\leq&C_0,\\
\lnm{B}&\leq&C_0.
\end{eqnarray}
\end{proof}

\subsection{Estimate of $E_1$ and $f^{\a}$}

\begin{lemma}\label{lemma 2}
We have
\begin{eqnarray}
\lnm{E_1}+\lnm{f^{\a}}&\leq&C_0.
\end{eqnarray}
Also, for each $t>0$, there exists $Q(t)$ (independent of $c$) such that $f^{\a}(s,x,p)=0$ for any $\abs{p}\geq Q(t)$ and $0\leq s\leq t$.
\end{lemma}
\begin{proof}
We divide the proof into several steps:\\
\ \\
Step 1: Truncated system.\\
Define a $C^{\infty}$ cut-off function $\psi: \r\rt[0,1]$ satisfying
\begin{eqnarray}
\psi(x)=\left\{
\begin{array}{ll}
1&\ \ \text{if}\ \ x\leq0,\\
0&\ \ \text{if}\ \ x\geq1.
\end{array}
\right.
\end{eqnarray}
For $L>R_0$, define
\begin{eqnarray}
f^{\a,L}_0(x,p)&=&f^{\a}_0(x,p)\psi(\abs{x}-L),\\
E^L_{2,0}(x)&=&E_{2,0}(x)\psi(\abs{x}-L),\\
B^L_{0}(x)&=&B_{0}(x)\psi(\abs{x}-L),\\
\rho_0^L(x)&=&\sum_{\a}e^{\a}\int_{\r^2}f^{\a,L}_0(x,p)\ud{p}=\rho_0(x)\psi(\abs{x}-L),\\
j_0^L(x)&=&\sum_{\a}e^{\a}\int_{\r^2}f^{\a,L}_0(x,p)V^{\a}(p)\ud{p}=j_0(x)\psi(\abs{x}-L),\\
E_{1,0}^L(x)&=&2\pi\int_{-\infty}^x\rho_0^L(y)\ud(y)-2\pi\int^{\infty}_x\rho_0^L(y)\ud(y).
\end{eqnarray}
We can directly verify the truncated data satisfy the compatibility condition
\begin{eqnarray}
\p_xE_{1,0}^L(x)=4\pi\rho_0^L(x).
\end{eqnarray}
The truncated initial data $f_0^{\a,L}(x,p)$, $E_{1,0}^L(x)$, $E_{2,0}^L(x)$, $B_0^L(x)$, have compact support both in space and momenta, so with fixed light speed $c$, we can apply the main theorem in \cite{10} to obtain a global smooth solution $f^{\a,L}$, $E_{1}^L$, $E_{2}^L$, $B^L$.\\
\ \\
Step 2: Characteristics.\\
Define the maximum velocity support of $f^{\a,L}$ as
\begin{eqnarray}
Q^L(t)&=&\sup\bigg\{\abs{p}: f^{\a,L}(t,x,p)\neq0\ \text{for some}\ s\in[0,t],x\in\r\bigg\},
\end{eqnarray}
and characteristics $X^{\a,L}(s;t,x,p)$, $P^{\a,L}(s;t,x,p)$ of the truncated system by
\begin{eqnarray}\label{wt 23}
\left\{
\begin{array}{rcl}
\dfrac{\ud{X^{\a,L}}}{\ud{s}}&=&V^{\a}_1(P^{\a,L}),\\\rule{0ex}{2.0em}
\dfrac{\ud{P_1^{\a,L}}}{\ud{s}}&=&e^{\a}\bigg(E_1^{L}(s,X^{\a,L})+\c^{-1}V^{\a}_2(P^{\a,L})B^L(s,X^{\a,L})\bigg),\\\rule{0ex}{2.0em}
\dfrac{\ud{P_2^{\a,L}}}{\ud{s}}&=&e^{\a}\bigg(E_2^{L}(s,X^{\a,L})-\c^{-1}V^{\a}_1(P^{\a,L})B^L(s,X^{\a,L})\bigg),\\\rule{0ex}{2.0em}
X^{\a,L}(t;t,x,p)&=&x,\\
P_1^{\a,L}(t;t,x,p)&=&p_1,\\
P_2^{\a,L}(t;t,x,p)&=&p_2,\\
\end{array}
\right.
\end{eqnarray}
Since
\begin{eqnarray}
\abs{V^{\a}(P^{\a,L})}\leq c,
\end{eqnarray}
then for any $\abs{x}\geq L+ct$, we have
\begin{eqnarray}
f^{\a,L}(t,x,p)=E_{1}^L(t,x)=E_{2}^L(t,x)=B^L(t,x)=\rho^L(t,x)=j^L(t,x)=0,
\end{eqnarray}
which means they are still compactly supported in space for any $t$. \\
\ \\
Step 3: Estimate of $E_1^L$.\\
Integrating the Vlasov equation over $p\in\r^2$ and summing up over $\a$, we obtain
\begin{eqnarray}
\dt\rho^L+\p_xj^L=0.
\end{eqnarray}
Since $\rho^L$ and $j^L$ are of compact support, we can further integrate over $x\in\r$ to obtain (since $L>R_0$)
\begin{eqnarray}
\int_{\r}\rho^L(t,x)\ud{x}=\int_{\r}\rho^L_0(x)\ud{x}=\int_{\r}\rho_0(x)\ud{x}=0.
\end{eqnarray}
Hence, from the equation $\p_xE_{1,0}^L(x)=4\pi\rho_0^L(x)$, we obtain
\begin{eqnarray}
E_{1}^L(t,x)&=&2\pi\int_{-\infty}^x\rho^L(t,y)\ud{y}-2\pi\int^{\infty}_x\rho^L(t,y)\ud{y}=4\pi\int_{-\infty}^x\rho^L(t,y)\ud{y}\\
&=&4\pi\int_{\r}\rho^L(t,y)\id{y\leq x}\ud{y},\no
\end{eqnarray}
Therefore, we have
\begin{eqnarray}\label{wt 21}
E_1^L(t,x)-E^L_{1,0}(x)&=&4\pi\int_{\r}\bigg(\rho^L(t,y)-\rho_0^L(y)\bigg)\id{y\leq x}\ud{y}\\
&=&4\pi\sum_{\a}e^{\a}\int_{\r}\int_{\r^2}\bigg(f^{\a,L}(t,y,p)-f^{\a,L}_0(y,p)\bigg)\id{y\leq x}\ud{p}\ud{y}.\no
\end{eqnarray}
Define the substitution $(y,p)\rt(\tilde y,\tilde p)$ as
\begin{eqnarray}
\left\{
\begin{array}{rcl}
\tilde y&=&X^{\a,L}(0;t,y,p),\\
\tilde p&=&P^{\a,L}(0;t,y,p).
\end{array}
\right.
\end{eqnarray}
It is a classical result that the Jacobian of this substitution is $1$.
Then we have
\begin{eqnarray}\label{wt 22}
\int_{\r}\int_{\r^2}f^{\a,L}(t,y,p)\id{y\leq x}\ud{p}\ud{y}&=&\int_{\r}\int_{\r^2}f^{\a,L}_0(\tilde y,\tilde p)\id{y\leq x}\ud{\tilde p}\ud{\tilde y}\\
&=&\int_{\r}\int_{\r^2}f^{\a,L}_0(\tilde y,\tilde p)\id{X^{\a,L}(t;0,\tilde y,\tilde p)\leq x}\ud{\tilde p}\ud{\tilde y}.\no
\end{eqnarray}
Substituting (\ref{wt 22}) into (\ref{wt 21}) and rewriting the dummy variables in (\ref{wt 22}) as $(y,p)$ we get
\begin{eqnarray}\label{wt 26}
E_1^L(t,x)-E^L_{1,0}(x)&=&4\pi\sum_{\a}e^{\a}\int_{\r}\int_{\r^2}f^{\a,L}_0(y,p)\bigg(\id{X^{\a,L}(t;0,y,p)\leq x}-\id{y\leq x}\bigg)\ud{p}\ud{y}.
\end{eqnarray}
For $(y,p)$ such that $f_0^{\a,L}(y,p)\neq0$, we have
\begin{eqnarray}
\abs{P^{\a,L}(s;0,y,p)}&\leq&Q^L(s),
\end{eqnarray}
which further implies
\begin{eqnarray}
\abs{V^{\a}\bigg(P^{\a,L}(s;0,y,p)\bigg)}&\leq&\abs{\frac{P^{\a,L}(s;0,y,p)}{m^{\a}}}\leq \frac{Q^L(s)}{m^{\a}}.
\end{eqnarray}
Therefore, we get the bound for maximal distance between initial position and position at time $t$,
\begin{eqnarray}
\abs{X^{\a,L}(t;0,y,p)-y}&=&\abs{\int_0^tV^{\a}_1\bigg(P^{\a,L}(s;0,y,p)\bigg)\ud{s}}\leq \frac{1}{m^{\a}}\int_0^tQ^L(s)\ud{s}.
\end{eqnarray}
We decompose the integral over $y\in\r$ in (\ref{wt 26}) to get
\begin{eqnarray}
E_1^L(t,x)-E^L_{1,0}(x)&=&I+II+III,
\end{eqnarray}
where $I$ is the integral over $y\in\bigg(x+\dfrac{1}{m^{\a}}\displaystyle\int_0^tQ^L(s)\ud{s},+\infty\bigg)$, $II$ is the integral over $y\in\bigg(-\infty, x-\dfrac{1}{m^{\a}}\displaystyle\int_0^tQ^L(s)\ud{s}\bigg)$, and $III$ is the integral over $y\in\bigg[x-\dfrac{1}{m^{\a}}\displaystyle\int_0^tQ^L(s)\ud{s}, x+\dfrac{1}{m^{\a}}\int_0^tQ^L(s)\ud{s}\bigg]$.
\begin{itemize}
\item
In the integral $I$, we have $\id{X^{\a,L}(t;0,y,p)\leq x}-\id{y\leq x}=1-1=0$, which implies $I=0$;
\item
In the integral $II$, we have $\id{X^{\a,L}(t;0,y,p)\leq x}-\id{y\leq x}=0-0=0$, which implies $II=0$;
\end{itemize}
Therefore, we have
\begin{eqnarray}
|E_1^L(t,x)-E^L_{1,0}(x)|&=&|III|\leq 4\pi\sum_{\a}|e^{\a}|\max_{\a,x,p}{f_0^{\a,L}}\bigg(\frac{2}{m^{\a}}\int_0^tQ^L(s)\ud{s}\bigg)(2Q_0)^2\leq C_0\int_0^tQ^L(s)\ud{s},
\end{eqnarray}
which means we may take the supremum to obtain
\begin{eqnarray}\label{wt 24}
\sup_{[0,t]\times\r}|E_1^L(s,x)|\leq C_0+C_0\int_0^tQ^L(s)\ud{s}.
\end{eqnarray}
On the other hand, based on the characteristic equation (\ref{wt 23}) and Lemma \ref{lemma 1}, we know for $t\in[0,T]$,
\begin{eqnarray}\label{wt 25}
Q^L(t)-Q_0\leq C_0\int_0^t\sup_x{\bigg(\abs{E_1^L(s,x)}+\abs{E_2^L(s,x)}+\abs{B^L(s,x)}\bigg)}\ud{s}\leq C_0+C_0\sup_{[0,t]\times\r}|E_1^L(s,x)|.
\end{eqnarray}
Combining (\ref{wt 24}) and (\ref{wt 25}), we obtain
\begin{eqnarray}
Q^L(t)\leq C_0+C_0\int_0^tQ^L(s)\ud{s}.
\end{eqnarray}
By Gronwall's inequality, we have for $t\in[0,T]$,
\begin{eqnarray}
0<Q^L(t)\leq C_0,
\end{eqnarray}
which further implies
\begin{eqnarray}
\lnm{E_1^L}\leq C_0,
\end{eqnarray}
where $C_0$ only depends on $T$ and the initial data and is independent of $L$ and $c$.\\
\ \\
Step 4: Synthesis: \\
For any finite $c$, the domain of dependence of the point $(t,x,p)$ is bounded in $[0,T]\times\r\times\r^2$, so we can always take $L$ large enough that the solution $f^{\a}(t,x,p)=f^{\a,L}(t,x,p)$. Hence, we have shown
\begin{eqnarray}
\lnm{E_1}\leq C_0.
\end{eqnarray}
The existence of $Q(t)$ is guaranteed by the analysis of $Q^L(t)$.
\end{proof}

\noindent The proof of Theorem \ref{main 1} is now complete.

\section{The Limit as c Tends to Infinity}

Regarding $f^{\a,\infty}, \rho^{\infty}, j^{\infty}, E^{\infty}_1$ as defined in (\ref{wt j11}), we have the following:
\begin{theorem} \label{limiting bounds}
We have
\begin{eqnarray}
\lnm{E_1^{\infty}}+\lnm{f^{\a,\infty}}+\lnm{\partial_xf^{\a,\infty}}+\\\lnm{\partial_{p_1}f^{\a,\infty}}+
\lnm{\partial_{p_2}f^{\a,\infty}}&\leq&C_0.\no
\end{eqnarray}
Also, for each $t>0$, there exists $Q^{\infty}(t)$ such that $f^{\a,\infty}(s,x,p)=0$ for any $\abs{p}\geq Q^{\infty}(t)$ and $0\leq s\leq t \leq T$.
\end{theorem}
\begin{proof}

The bound on $\lnm{E_1^{\infty}}+\lnm{f^{\a,\infty}}$ and the construction of $Q^{\infty}(t)$ may be obtained by the methods of the previous section.  To bound the derivatives define $R_x^{\a}=\p_{x}f^{\a,\infty}$, $R_{p_1}^{\a}=\p_{p_1}f^{\a,\infty}$ and $R_{p_2}^{\a}=\p_{p_2}f^{\a,\infty}$. Then they satisfy
\begin{eqnarray}
\dt R_x^{\a}+V^{\a,\infty}_{1}(p)\p_{x}R_x^{\a}
+e^{\a}E_1^{\infty}\p_{p_1}R_x^{\a}=-e^{\a}\p_xE_1^{\infty}R_{p_1}^{\a},\\
\dt R_{p_1}^{\a}+V^{\a,\infty}_{1}(p)\p_{x}R_{p_1}^{\a}
+e^{\a}E_1^{\infty}\p_{p_1}R_{p_1}^{\a}=-\p_{p_1}V^{\a,\infty}_{1}(p)R_x^{\a},\\
\dt R_{p_2}^{\a}+V^{\a,\infty}_{1}(p)\p_{x}R_{p_2}^{\a}
+e^{\a}E_1^{\infty}\p_{p_1}R_{p_2}^{\a}=0,
\end{eqnarray}
with initial data
\begin{eqnarray}
R_x^{\a}(0,x,p)&=&\p_xf_0^{\a},\\
R_{p_1}^{\a}(0,x,p)&=&\p_{p_1}f_0^{\a},\\
R_{p_2}^{\a}(0,x,p)&=&\p_{p_2}f_0^{\a}.
\end{eqnarray}
By the bound on $f^{\a,\infty}$ and using $Q^{\infty}(t)$ we have
\begin{eqnarray}
\abs{\p_xE_1^{\infty}}=\abs{4\pi\rho^{\infty}}\leq C_0.
\end{eqnarray}
So integration along the characteristics of (\ref{wt j11}) yields
\begin{eqnarray}
\abs{R_x^{\a}(t,x,p)}&\leq&C_0+C_0\int_0^t\sup_{[0,s]\times\mathbb{R}\times\mathbb{R}^2}\abs{R_{p_1}^{\a}}\ud{s},\label{wt 27}\\
\abs{R_{p_1}^{\a}(t,x,p)}&\leq&C_0+C_0\int_0^t\sup_{[0,s]\times\mathbb{R}\times\mathbb{R}^2}\abs{R_x^{\a}}\ud{s},\label{wt 28}\\
\abs{R_{p_2}^{\a}(t,x,p)}&\leq&C_0.
\end{eqnarray}
Combining (\ref{wt 27}) and (\ref{wt 28}) and yields
\begin{eqnarray}
\sup_{[0,t]\times\mathbb{R}\times\mathbb{R}^2}\abs{R_x^{\a}}&\leq&C_0+C_0t+C_0t\int_0^t\sup_{[0,s]\times\mathbb{R}\times\mathbb{R}^2}\abs{R_x^{\a}}\ud{s}\\
&\leq&C_0+C_0\int_0^t\sup_{[0,s]\times\mathbb{R}\times\mathbb{R}^2}\abs{R_x^{\a}}\ud{s}.\no
\end{eqnarray}
By Gronwall's inequality, we have
\begin{eqnarray}
\sup_{[0,t]\times\mathbb{R}\times\mathbb{R}^2}\abs{R_x^{\a}}&\leq&C_0.
\end{eqnarray}
Similarly, we can prove
\begin{eqnarray}
\sup_{[0,t]\times\mathbb{R}\times\mathbb{R}^2}\abs{R_{p_1}^{\a}}&\leq&C_0.
\end{eqnarray}
\end{proof}

\subsection{Estimate of $E_2$ and $B$}

\begin{lemma}\label{lemma 3}
We have
\begin{eqnarray}
\lnm{E_2}+\lnm{B}\leq C_0\c^{-1}.
\end{eqnarray}
\end{lemma}
\begin{proof}
We divide the proof into several steps:\\
\ \\
Step 1: Setup.\\
Define
\begin{eqnarray}
D_0=R_0+(1+T)\sup_{[0,T]}(Q(t)+Q^{\infty}(t))
\end{eqnarray}
and note that $|x|>D_0$ and $f^{\a}(t,x,p)\neq 0$ implies
\begin{eqnarray}
f^{\a}(t,x,p)=F^{\a}(X^{\a}(0;t,x,p)).
\end{eqnarray}
Denote $m_0=\min_{\a}m^{\a}$ and note that $f^{\a}(t,x,p)\neq0$ implies $\abs{V^{\a}}\leq\dfrac{D_0}{m_0}$.
Let
\begin{eqnarray}
\l=\left\{(t,x): t\in[0,T], \abs{x}\geq D_0\left(1+\frac{t}{m_0}\right)\right\}.
\end{eqnarray}
Thus, if $(t,x)\in\l$, with $s\in[0,t]$ and $f^{\a}(t,x,p)\neq0$, then
\begin{eqnarray}
\abs{X^{\a}(s;t,x,p)}&\geq&\abs{x}-\int_s^t\abs{V_1^{\a}\bigg(P^{\a}(\tau;t,x,p)\bigg)}\ud{\tau}\geq\abs{x}-\frac{D_0}{m_0}(t-s)\\
&\geq&D_0\left(1+\frac{t}{m_0}\right)-\frac{D_0}{m_0}(t-s)=D_0\left(1+\frac{s}{m_0}\right),\no
\end{eqnarray}
which implies $\left(s, X^{\a}(s;t,x,p)\right)\in\l$.
Denote
\begin{eqnarray}
\nm{\sigma(t)}_{\l}=\sup\left\{\abs{\sigma(s,x)}: s\in[0,t], (s,x)\in\l\right\}.
\end{eqnarray}
\ \\
Step 2: Estimate of $j$ in $\l$.\\
Consider $(t,x,p)$ with $(t,x)\in\l$ and $f^{\a}(t,x,p)\neq 0$. Then
\begin{eqnarray}
\dfrac{\ud{P_1^{\a}}}{\ud{s}}&=&e^{\a}\bigg(E_1(s,X^{\a})+\c^{-1}V^{\a}_2(P^{\a})B(s,X^{\a})\bigg),\\
\dfrac{\ud{P_2^{\a}}}{\ud{s}}&=&e^{\a}\bigg(E_2(s,X^{\a})-\c^{-1}V^{\a}_1(P^{\a})B(s,X^{\a})\bigg),
\end{eqnarray}
which implies
\begin{eqnarray}
\abs{\dfrac{\ud{P_1^{\a}}}{\ud{s}}}&\leq&C_0\bigg(\nm{E_1}_{\l}+\nm{B}_{\l}\bigg),\\
\abs{\dfrac{\ud{P_2^{\a}}}{\ud{s}}}&\leq&C_0\bigg(\nm{E_2}_{\l}+\nm{B}_{\l}\bigg).
\end{eqnarray}
Therefore,
\begin{eqnarray}
\abs{P^{\a}(0;t,x,p)-p}\leq C_0\bigg(\nm{E_1}_{\l}+\nm{E_2}_{\l}+\nm{B}_{\l}\bigg),
\end{eqnarray}
and
\begin{eqnarray}
\abs{f^{\a}(t,x,p)-F^{\a}(p)}&=&\abs{F^{\a}(P^{\a}(0;t,x,p))-F^{\a}(p)}\leq C_0\bigg(\nm{E_1}_{\l}+\nm{E_2}_{\l}+\nm{B}_{\l}\bigg).
\end{eqnarray}
It follows from assumption (\ref{wt j7}) that
\begin{eqnarray}
\displaystyle\abs{\sum_{\a}\bigg(e^{\a}\int_{\r^2}F^{\a}V^{\a}\ud{p}\bigg)}=\abs{\sum_{\a}\bigg(e^{\a}\int_{\r^2}F^{\a}(V^{\a}-p/m^{\a})\ud{p}\bigg)}\\\leq \sum_{\a}|e^{\a}|\int_{\mathbb{R}^2}F^{\a}c^{-2}\frac{|p|^3}{(m^{\a})^3}dp\leq C_0\c^{-2},\no
\end{eqnarray}
so
\begin{eqnarray}\label{wt j18}
\abs{j(t,x)}&=&\abs{\sum_{\a}\bigg(e^{\a}\int_{\r^2}f^{\a}V^{\a}\ud{p}\bigg)}\leq C_0\c^{-2}+\abs{\sum_{\a}\bigg(e^{\a}\int_{\r^2}(f^{\a}-F^{\a})V^{\a}\ud{p}\bigg)}\\
&\leq&C_0\c^{-2}+C_0\bigg(\nm{E_1}_{\l}+\nm{E_2}_{\l}+\nm{B}_{\l}\bigg).\no
\end{eqnarray}
Hence, we know
\begin{eqnarray}
\nm{j(t)}_{\l}&\leq&C_0\c^{-2}+C_0\bigg(\nm{E_1(t)}_{\l}+\nm{E_2(t)}_{\l}+\nm{B(t)}_{\l}\bigg).
\end{eqnarray}
\ \\
Step 3: Estimate of $E_1$ in $\l$.\\
Since $E_1$ satisfies $\dt E_1=-4\pi j_1$, we have
\begin{eqnarray}
E_1(t,x)=E_1(0,x)-4\pi\int_0^tj_1(s,x)\ud{s}.
\end{eqnarray}
For $(t,x)\in\l$, we know $E_1(0,x)=0$. Thus,
\begin{eqnarray}
\abs{E_1(t,x)}\leq 4\pi\int_0^t\abs{j_1(s,x)}\ud{s}.
\end{eqnarray}
Therefore, we have
\begin{eqnarray}\label{wt j15}
\nm{E_1(t)}_{\l}\leq C_0\c^{-2}+C_0\int_0^t\bigg(\nm{E_1(s)}_{\l}+\nm{E_2(s)}_{\l}+\nm{B(s)}_{\l}\bigg)\ud{s}.
\end{eqnarray}
\ \\
Step 4: Estimate of $E_2$ and $B$ in $\l$.\\
Next, we consider $E_2$ and $B$ for any $(t,x)\in[0,T]\times\r$. Based on (\ref{wt j17}), we can directly obtain
\begin{eqnarray}
E_2&=&-2\pi\int_0^t\bigg(j_2(\tau,x-\c(t-\tau))+j_2(\tau,x+\c(t-\tau))\bigg)\ud{\tau},\\
B&=&-2\pi\int_0^t\bigg(j_2(\tau,x-\c(t-\tau))-j_2(\tau,x+\c(t-\tau))\bigg)\ud{\tau}.
\end{eqnarray}
Hence, in order to estimate $E_2$ and $B$, we need to bound $\abs{\displaystyle\int_0^tj_2(\tau,x-\c(t-\tau))\ud{\tau}}$ and $\abs{\displaystyle\int_0^tj_2(\tau,x+\c(t-\tau))\ud{\tau}}$.
By substitution, we need to bound
\begin{eqnarray}
\int_0^tj_2(\tau,x-\c(t-\tau))\ud{\tau}=\c^{-1}\int_{x-\c t}^{x} j_2\left(t-\frac{x-y}{\c},y\right)\ud{y},
\end{eqnarray}
and
\begin{eqnarray}
\int_0^tj_2(\tau,x+\c(t-\tau))\ud{\tau}=\c^{-1}\int^{x+\c t}_{x} j_2\left(t-\frac{y-x}{\c},y\right)\ud{y}.
\end{eqnarray}
Note that $(s,y)\notin \l \Rightarrow |y|<D_0(1+s/m_0)\leq D_0(1+T/m_0)=C_0$ so by (\ref{wt j18}) and Theorem \ref{main 1}
\begin{eqnarray}
\c^{-1}\abs{\int_{x-\c t}^{x} j_2\left(t-\frac{x-y}{\c},y\right)\ud{y}}&\leq&c^{-1}\int_{-C_0}^{C_0}\nm{j_2}_{L^{\infty}}dy\\
&+&c^{-1}\int_{x-ct}^xC_0\bigg(c^{-2}+\nm{E_1(t-\frac{x-y}{c})}_{\l}+\nm{E_2(t-\frac{x-y}{c})}_{\l}+\nm{B(t-\frac{x-y}{c})}_{\l}\bigg)\ud{y}\no\\
&\leq&C_0\c^{-1}+C_0\int_0^t\bigg(\nm{E_1(s)}_{\l}+\nm{E_2(s)}_{\l}+\nm{B(s)}_{\l}\bigg)\ud{s}\no,
\end{eqnarray}
and similarly
\begin{eqnarray}
\abs{\c^{-1}\int^{x+\c t}_{x} j_2\left(t-\frac{y-x}{\c},y\right)\ud{y}}&\leq&C_0\c^{-1}+C_0\int_0^t\bigg(\nm{E_1(s)}_{\l}+\nm{E_2(s)}_{\l}+\nm{B(s)}_{\l}\bigg)\ud{s}.
\end{eqnarray}
This implies
\begin{eqnarray}\label{wt j13}
|E_2(t,x)|+|B(t,x)|\leq C_0\c^{-1}+C_0\int_0^t\bigg(\nm{E_1(s)}_{\l}+\nm{E_2(s)}_{\l}+\nm{B(s)}_{\l}\bigg)\ud{s}.
\end{eqnarray}
We may restrict (\ref{wt j13}) to $\l$ to obtain
\begin{eqnarray}\label{wt j16}
\nm{E_2(t)}_{\l}+\nm{B(t)}_{\l}\leq C_0\c^{-1}+C_0\int_0^t\bigg(\nm{E_1(s)}_{\l}+\nm{E_2(s)}_{\l}+\nm{B(s)}_{\l}\bigg)\ud{s}.
\end{eqnarray}
\ \\
Step 5: Synthesis.\\
In summary, we know from (\ref{wt j15}) and (\ref{wt j16}) that
\begin{eqnarray}
\nm{E_1(t)}_{\l}+\nm{E_2(t)}_{\l}+\nm{B(t)}_{\l}\leq C_0\c^{-1}+C_0\int_0^t\bigg(\nm{E_1(s)}_{\l}+\nm{E_2(s)}_{\l}+\nm{B(s)}_{\l}\bigg)\ud{s}.
\end{eqnarray}
By Gronwall's inequality, we have
\begin{eqnarray}\label{wt j14}
\nm{E_1(t)}_{\l}+\nm{E_2(t)}_{\l}+\nm{B(t)}_{\l}\leq C_0\c^{-1}\ue^{C_0t}\leq C_0\c^{-1}.
\end{eqnarray}
Therefore, by (\ref{wt j13}) we have
\begin{eqnarray}
\lnm{E_2(t,x)}+\lnm{B(t,x)}\leq C_0\c^{-1}.
\end{eqnarray}

\end{proof}

\subsection{Estimate of $E_1-E_1^{\infty}$ and $f^{\a}-f^{\a}_{\infty}$}

\begin{lemma}\label{lemma 4}
We have
\begin{eqnarray}
\lnm{E_1-E_1^{\infty}}+\lnm{f^{\a}-f^{\a}_{\infty}}\leq C_0\c^{-1}.
\end{eqnarray}
\end{lemma}
\begin{proof}

Since $\dt E_1=-4\pi j_1$ and $\dt E_1^{\infty}=-4\pi j_1^{\infty}$, we have
\begin{eqnarray}\label{ct 01}
\abs{E_1-E_1^{\infty}}&\leq&4\pi\int_0^t\abs{j_1(s,x)-j_1^{\infty}(s,x)}\ud{s}\\
&\leq& C_0\sum_{\a}\int_0^t\int_{\r^2}\abs{h^{\a}}\ud{p}\ud{s}\no
\end{eqnarray}
where $h^{\a}(t,x,p)=f^{\a}(t,x,p)-f^{\a,\infty}(t,x,p)$.  Note that $h^{\a}$ satisfies the equation
\begin{eqnarray}
&&\dt h^{\a}+V^{\a}_1(p)\p_{x}h^{\a}
+e^{\a}\left(E_1+\c^{-1}V^{\a}_2(p)B\right)\p_{p_1}h^{\a}+e^{\a}\left(E_2-\c^{-1}V^{\a}_1(p)B\right)\p_{p_2}h^{\a}\\
&=&(V^{\a,\infty}_1(p)-V^{\a}_1(p))\p_xf^{\a,\infty}+e^{\a}\left(E^{\infty}_1-E_1-\c^{-1}V^{\a}_2(p)B\right)\p_{p_1}f^{\a,\infty}-e^{\a}
\left(E_2-\c^{-1}V^{\a}_1(p)B\right)\p_{p_2}f^{\a,\infty}\no
\end{eqnarray}
and
\begin{eqnarray}
h^{\a}(0,x,p)=0.
\end{eqnarray}
By Theorem \ref{limiting bounds} and by Lemma \ref{lemma 3} we have
\begin{eqnarray}
\abs{\bigg(V^{\a,\infty}_1(p)-V^{\a}_1(p)\bigg)\p_xf^{\a,\infty}}&\leq&C_0\c^{-2},\\
\abs{e^{\a}\left(E_1^{\infty}-E_1-\c^{-1}V^{\a}_2(p)B\right)\p_{p_1}f^{\a,\infty}}&\leq&C_0\c^{-1}+\sup_{[0,t]\times\r}\abs{E_1-E_1^{\infty}},\\
\abs{e^{\a}\left(E_2-\c^{-1}V^{\a}_1(p)B\right)\p_{p_2}f^{\a,\infty}}&\leq&C_0\c^{-1}.
\end{eqnarray}
Then integrating along the characteristics of the Vlasov equation, we have
\begin{eqnarray}\label{ct 02}
\abs{h^{\a}(t,x,p)}\leq C_0\c^{-1}+C_0\int_0^t\sup_{[0,s]\times\r}\abs{E_1-E_1^{\infty}}\ud{s}.
\end{eqnarray}
Combining (\ref{ct 01}) and (\ref{ct 02}), we have
\begin{eqnarray}
\abs{E_1-E_1^{\infty}}&\leq& C_0\sum_{\a}\int_0^t\int_{|p|<Q(s)+Q^{\infty}(s)}\sup_{[0,s]\times\mathbb{R}\times\mathbb{R}^2}|h^{\a}|dpds\\
&\leq&C_0t(c^{-1}+\int_0^t\sup_{[0,s]\times\mathbb{R}}\abs{E_1-E_1^{\infty}}\ud{s}).\no
\end{eqnarray}
Taking the supremum yields
\begin{eqnarray}
\sup_{[0,t]\times\r}\abs{E_1-E_1^{\infty}}&\leq& C_0T(\c^{-1}+\int_0^t\sup_{[0,s]\times\r}\abs{E_1-E_1^{\infty}}\ud{s}).
\end{eqnarray}
Using Gronwall's inequality, for $t\in[0,T]$ we obtain
\begin{eqnarray}
\sup_{[0,t]\times\r}\abs{E_1-E_1^{\infty}}&\leq&C_0\c^{-1},
\end{eqnarray}
which is
\begin{eqnarray}
\lnm{E_1-E_1^{\infty}}&\leq&C_0\c^{-1}.
\end{eqnarray}
By (\ref{ct 02}) this implies
\begin{eqnarray}
\lnm{f^{\a}-f^{\a}_{\infty}}&\leq&C_0\c^{-1}.
\end{eqnarray}

\end{proof}
\noindent The proof of Theorem \ref{main 2} is now complete.

\end{document}